\documentclass[12pt]{amsart}
\usepackage[osf,sc]{mathpazo}
\usepackage{amssymb}

\usepackage{geometry}
\usepackage{bm}
\usepackage{graphicx}
\usepackage{hyperref}
\usepackage{varwidth}
\hypersetup{
	colorlinks=true, 
	linktoc=all,     
	linkcolor=blue,
	citecolor=red,
	filecolor=black,
	urlcolor=blue	
}
\usepackage{enumerate}
\usepackage[inline]{enumitem}
\makeatletter
\newcommand{\inlineitem}[1][]{%
	\ifnum\enit@type=\tw@
	{\descriptionlabel{#1}}
	\hspace{\labelsep}%
	\else
	\ifnum\enit@type=\z@
	\refstepcounter{\@listctr}\fi
	\quad\@itemlabel\hspace{\labelsep}%
	\fi} \makeatother
\newcommand{\gb}{\beta}

\newcommand{\gd}{\delta}

\newcommand{\gl}{\lambda}
\newcommand{\gm}{\mu}
\newcommand{\gn}{\nu}

\newcommand{\gp}{\pi}

\newcommand{\gt}{\tau}

\newcommand{\gf}{\phi}

\newcommand{\Gs}{\Sigma}

\newcommand{\Gom}{\Omega}

\newcommand{\subs}{\subset}

\newcommand{\bs}{\backslash}

\newcommand{\nin}{\notin}

\newcommand{\mbb}{\mathbb}

\newcommand{\mcl}{\mathcal}

\newcommand{\us}{\underset}
\newcommand{\os}{\overset}

\newcommand{\lra}{\longrightarrow}

\newcommand{\Z}{\mbb Z}

\newcommand{\es}{\emptyset}

\newcommand{\equ}[1]{%
	\begin{equation*}
		#1
	\end{equation*}
}
\newcommand{\equa}[1]{%
	\begin{equation*}
		\begin{aligned}
			#1
		\end{aligned}
	\end{equation*}
}

\DeclareMathOperator{\Det}{Det}

\DeclareMathOperator{\Aut}{Aut}

\newcommand{\matthree}[9]{%
	\begin{pmatrix}
		#1 & #2 & #3\\ #4 & #5 & #6\\ #7 & #8 & #9
	\end{pmatrix}
}

\newtheorem{theorem}{Theorem}[section]

\newtheorem{ques}[theorem]{Question}

\theoremstyle{definition}
\newtheorem{defn}[theorem]{Definition}

\newtheorem{prob}[theorem]{Problem}
\theoremstyle{remark}

\numberwithin{equation}{section}
\makeatletter
\def\namedlabel#1#2{\begingroup
	\def\@currentlabel{#2}%
	\label{#1}\endgroup
}
\makeatother

\newtheorem*{thmOmega}{\bf{Theorem} $\bm{\Gom}$}

\begin{document}
\title[Antipodal Point Arrangements on Spheres]
{Antipodal Point Arrangements on Spheres and Classification of Normal Systems}
\author[C.P. Anil Kumar]{Author: C.P. Anil Kumar*}
\address{Center for Study of Science, Technology and Policy
	\# 18 \& \#19, 10th Cross, Mayura Street,
	Papanna Layout, Nagashettyhalli, RMV II Stage,
	Bengaluru - 560094
	Karnataka,INDIA}
\email{akcp1728@gmail.com}
\thanks{*The author is supported by a research grant and facilities provided by Center for study of Science, 
	Technology and Policy (CSTEP), Bengaluru, INDIA for this research work.}
\begin{abstract}
For any positive integer $k>1$, we classify the antipodal point arrangements
on the sphere $S^k$ up to an isomorphism,
by associating a finite complete set of cycle invariants.
\end{abstract}
\maketitle

\section{\bf{Introduction}}
The main motivation to write this article, arises during the association of invariants such as normal systems to the infinity type hyperplane arrangements which is done in C.~P.~Anil Kumar~\cite{NRHA}. The infinity type line arrangements, their nomenclature and some of their properties are discussed in C.~P.~Anil Kumar~\cite{CPAK}.
Here we classify normal systems combinatorially. Before we restate the relevant Problem~\ref{prob:CNS} regarding classification of normal systems we need a few definitions.
 \begin{defn}[Normal System]
	\label{defn:NS}
	~\\
	Let $\mcl{N}=\{L_1,L_2,\ldots,L_n\}$ be a finite set of lines passing through the origin in $\mbb{R}^m$. Let $\mcl{U}=\{\pm v_1,\pm v_2,\ldots,\pm v_n\}$ be a set of antipodal pairs of non-zero vectors on these lines. We say that
	$\mcl{N}$ forms a normal system, if the set 
	\equ{\mcl{B}=\{v_1,v_2,\ldots,v_n\}}
	of vectors has the property that, any subset $\mcl{C}\subs \mcl{B}$ of cardinality at most $m$ is a linearly independent set. 
\end{defn}
\begin{defn}[Convex Positive Bijection and Isomorphic Normal Systems]
	\label{defn:CPB}
	~\\
	Let  \equ{\mcl{N}_1=\{L_1,L_2,\ldots,L_n\},\mcl{N}_2=\{M_1,M_2,\ldots,M_n\}} 
	be two finite sets of lines passing through the origin in $\mbb{R}^m$, both of them have the same 
	cardinality $n$, which form normal systems. Let 
	\equ{\mcl{U}_1=\{\pm v_1,\pm v_2,\ldots,\pm v_n\},\mcl{U}_2=\{\pm w_1,\pm w_2,\ldots,\pm w_n\}} 
	be two sets of antipodal pairs of vectors on these lines in $\mcl{N}_1,\mcl{N}_2$ respectively. We say a bijection 
	$\gd:\mcl{U}_1\lra \mcl{U}_2$ is a convex positive bijection if 
	\equ{\gd(-u)=-\gd(u),u\in \mcl{U}_1} and for any basis $\mcl{B}=\{u_1,u_2,\ldots,u_m\}\subs \mcl{U}_1$
	and a vector $u\in \mcl{U}_1$ we have 
	\equa{u&=\us{i=1}{\os{m}{\sum}}a_iu_i \text{ with }a_i>0, 1\leq i \leq m, \text{ if and only if },\\
		\gd(u)&=\us{i=1}{\os{m}{\sum}}b_i\gd(u_i) \text{ with }b_i>0, 1\leq i \leq m.}
	We say two normal systems are isomorphic if there exists a convex positive bijection between their 
	corresponding sets of antipodal pairs of vectors. 
\end{defn}
Now we mention the relevant open problem regarding classification of normal systems.
\begin{prob}[Classification of Normal Systems and Finding Representatives in Each Isomorphism Class]
\label{prob:CNS}
~\\
Let $n,m$ be positive integers.
Classify and enumerate the normal systems in $\mbb{R}^m$ up to an isomorphism by associating invariants which can be used to 
easily construct a family of normal systems representing each isomorphism class for every positive integer 
cardinality $n$ of the normal system.
\end{prob}
Here in this article we classify normal systems up to an isomorphism by associating a finite complete 
set of cycle invariants. The enumeration problem of the number of 
isomorphism classes of normal systems and the problem of representing their isomorphism classes 
by a well defined list of representatives still remain open (refer to Question~\ref{ques:ENUMREP}) and known only for an initial few values of $n$ for any $m$.
The exact statement of main Theorem~\ref{theorem:LCI} about classification of normal systems cannot be stated here as it requires more definitions and concepts
which at present are not motivated and developed. Hence we defer the statement to its appropriate Section~\ref{sec:APAHDSCNS} of the article.

\subsection{\bf{Brief Survey and the Structure of the Paper}}
With relevance to antipodal point arrangements (refer to Definition~\ref{defn:kDAPA}) or normal systems, the theory of matroids
is a well studied subject. Matroids are combinatorial abstractions of vector configurations and hyperplane arrangements.
Here in this article 
we study specific kind of antipodal pairs of vectors arranged on spheres, vector configurations,
which are associated to normal systems that arise
from hyperplane arrangements and classify them combinatorially. The method of associating cycle invariants
as a combinatorial model to point arrangements in 
the plane has already been explored by authors J.E.Goodman and R.Pollack~\cite{MR0583961}.
Also the slope problem mentioned in chapter $10$, page $60$ in M.~Aigner and G.~M.~Ziegler~\cite{MR3823190},
Proofs from THE BOOK, explains a similar method.

Section~\ref{sec:APA2D} is devoted to the classification of antipodal point arrangements on $S^2$ in 
two dimensions. Theorem~\ref{theorem:MITTD} states the classification theorem in the case of dimension two.
Section~\ref{sec:NonIsoNS} revisits the two non-isomorphic examples of normal systems in dimension three
that are mentioned in C.~P.~Anil~Kumar~\cite{NRHA} and computes the combinatorial invariants.
Sections~[\ref{sec:APAHDSCNS}-\ref{sec:FMT}] are devoted to classification of antipodal point arrangements on $S^k$ 
in higher dimensions for $k>2$. Theorem~\ref{theorem:LCI} in Section~\ref{sec:FMT} states the classification theorem in higher dimensions. 
In final Section~\ref{sec:ENUMREP} we pose open Question~\ref{ques:ENUMREP} about normal systems.

\section{\bf{Antipodal Point Arrangements on the $2\operatorname{-}$Sphere $S^2$}}
\label{sec:APA2D}
~\\
Now we define antipodal point arrangements on the $2\operatorname{-}$sphere $S^2$.
\begin{defn}[Antipodal Point Arrangement on the $2\operatorname{-}$Sphere $S^2$]
\label{defn:2DAPA}
~\\
We say a set $\mcl{P}_n=\{\pm P_1,\pm P_2,\ldots,\pm P_n\}\subs S^2$ of points is a point 
arrangement on the sphere if three points of $\mcl{P}_n$ are linearly dependent then some two of them are antipodal.
\end{defn}
\begin{defn}[Isomorphism Between two Antipodal Point Arrangements on the $2\operatorname{-}$Sphere $S^2$]
\label{defn:2DIso}
~\\
Two point arrangements 
\equ{\mcl{P}_n=\{\pm P_1,\pm P_2,\ldots,\pm P_n\},\mcl{Q}_m=\{\pm Q_1,\pm Q_2,\ldots,\pm Q_m\} \subs 
S^2} 
are isomorphic if $n=m$ and there is a bijection $\gf:\mcl{P}_n \lra \mcl{Q}_n$ between the two sets such that the 
following occurs.
\begin{itemize}
\item $\gf(-A)=-\gf(A)$ for all $A\in \mcl{P}_n$.
\item  for any $A,B,C,D\in \mcl{P}_n$ if $D$ is a positive combination of $A,B,C$ if and only if 
$\gf(D)$ is a positive combination of $\gf(A),\gf(B),\gf(C)$. 
\end{itemize}
We sometimes also say that the isomorphism $\gf$ is a convex positive bijection.
We say $\gf$ is orientation preserving if for any three points $A,B,C\in \mcl{P}_n$ the ordered triple $(A,B,C)$ has positive determinant if and only if the ordered triple $(\gf(A),\gf(B),\gf(C))$ has positive determinant. We say $\gf$ is orientation reversing if for any three points $A,B,C\in \mcl{P}_n$ the ordered triple $(A,B,C)$ has positive determinant if and only if the ordered triple $(\gf(A),\gf(B),\gf(C))$ has negative determinant.
\end{defn}
\begin{theorem}
\label{theorem:PreservingReversing}
Let \equ{\mcl{P}_n=\{\pm P_1,\pm P_2,\ldots,\pm P_n\},\mcl{Q}_n=\{\pm Q_1,\pm Q_2,\ldots,\pm Q_n\} \subs S^2} 
be two antipodal point arrangements on the $2\operatorname{-}$Sphere $S^2$. If $\gd:\mcl{P}_n \lra \mcl{Q}_n$ is an isomorphism then it is either an orientation preserving isomorphism or it is an orientation reversing isomorphism. 
\end{theorem}
\begin{proof}
Let $R_1,R_2,R_3\in \mcl{P}_n$ be linearly independent. Let $S_1=\gd(R_1),S_2=\gd(R_2),S_3=\gd(R_3)\in \mcl{Q}_n$. We show that the sign of the product \equ{\Det(R_1,R_2,R_3)\Det(S_1,S_2,S_3)} does not change for any such choice of three linearly independent elements $R_1,R_2,R_3$. If we change signs of $R_i$ for some $1\leq i\leq 3$ or interchange $R_i,R_j$ for some $1\leq i\neq j\leq 3$ and do the same operations with $S_i,1\leq i\leq 3$ then it is clear that the sign of product remains unchanged. It is also clear that if $R\in \mcl{P}_n$ and $R\nin \{\pm R_1,\pm R_2, \pm R_3\}$
then there exists $a,b,c\in \mbb{R}^{*}$ and $d,e,f\in \mbb{R}^{*}$ such that we have 
\equa{&sign(a)=sign(d),sign(b)=sign(e),sign(c)=sign(f)\text{ and }\\ &R=aR_1+bR_2+cR_3,S=\gd(R)=dS_1+eS_2+fS_3.}
Hence again the sign of \equa{\Det(aR_1+bR_2+cR_3,R_2,R_3)&\Det(dS_1+eS_2+fS_3,S_2,S_3)=\\
&\Det(R,R_2,R_3)\Det(S,S_2,S_3)} does not change. Here we have replaced $R_1$ by $R$ and $S_1$ by $S$. 
This way we have that the sign does not change for any such choice of elements $R_1,R_2,R_3$ which proves the theorem.
\end{proof}
\subsection{Algebraic Symbols Associated to a Four-Antipodal Point Arrangement on the Sphere $S^2$}
We begin this section with the standard arrangement.
\subsubsection{\bf{The Standard Arrangement and its Associated Symbols}}
\label{sec:SAAS}
The arrangement $\mcl{S}_4 \subs S^2$ consists of four antipodal pairs of points given by 
\equa{\mcl{S}_4=\{&x=[(1,0,0)],x=[(-1,0,0)],y=[(0,1,0)],-y=[(0,-1,0)],\\& z=[(0,1,0)], -z=[(0,-1,0)],
P \text{ \big(a point in } I\operatorname{-}Octant\big),\\&-P \text{ \big(antipode point of } P \text{ in } VII\operatorname{-}Octant\big)\} \subs S^2}
There are twenty four symbols that we associate to this standard arrangement. Before we actually describe these
symbols we mention four important aspects.
\begin{enumerate}
\item A symbol is of the form 
\equ{a\lra (b,c,d)}
\item We say that it is compatible or associated to an antipodal point arrangement if $a,b,c,d$ represent elements
of the arrangement in $S^2$ such that $a$ is a positive combination of $b,c,d$.
\item If we give an anticlockwise local orientation to the plane $a^{\perp}$ with the direction ray 
$a\in S^2$ representing the thumb then ignoring signs the line cycle is clockwise oriented and is given by 
\equ{(bcd), \text{ and not by }(bdc).}
For example to get the symbol $P\lra (y,x,z)$ for $\mcl{S}_4$ refer to the first octant view in Figure~\ref{fig:One}. Here the line cycle $(yxz)$ is obtained by moving clockwise around $P$. 
\item \label{enum:SymbolNegDet}
The ordered triple $(b,c,d)$ where $a$ is a positive combination of $b,c,d$ has negative determinant.
\end{enumerate}
The associated symbols for the standard arrangement $\mcl{S}_4$ are given by  
\equa{&P \lra (y,x,z), P \lra (x,z,y),P \lra (z,y,x),\\
&x\lra (-y,P,-z),x\lra (P,-z,-y),x\lra (-z,-y,P),\\
&z\lra (-x,P,-y),z\lra (P,-y,-x),z\lra (-y,-x,P),\\
&y\lra (-z,P,-x),y\lra (-x,-z,P),y\lra (P,-x,-z),\\
&-x\lra (-P,y,z),-x\lra (z,-P,y),-x\lra (y,z,-P),\\
&-z\lra (-P,x,y),-z\lra (x,y,-P),-z\lra (y,-P,x),\\
&-y\lra (-P,z,x),-y\lra (x,-P,z),-y\lra (z,x,-P),\\
&-P\lra (-x,-y,-z),-P\lra (-z,-x,-y),-P\lra (-y,-z,-x).}
In the above symbols the triples are all negatively (clockwise) oriented, that is, given that $P$ is in the first octant
the triples have determinant negative. 

\subsubsection{\bf{The Symmetry Group on Four Elements and its Action on Symbols}}
~\\
Here we explore the symmetry involved in the above set of $24$ compatible symbols. 
We state the following theorem on the action of the symmetry group $S_4$ on the set of symbols and 
describe the transitive orbits.
\begin{theorem}
\label{theorem:S4Action}
The group $S_4$ acts on the set 
\equa{S&=\{p\lra (q,r,s)\mid p,q,r,s \in \{\pm a,\pm b,\pm c,\pm d\}\text{ such that }\\
&\{\pm p\}\cap \{\pm q\}=\{\pm p\}\cap \{\pm r\}=\{\pm p\}\cap \{\pm s\}=\{\pm q\}\cap \{\pm r\}=\\
&\{\pm q\}\cap \{\pm s\}=\{\pm r\}\cap \{\pm s\}=\es\}}
of all symbols with the action given by
\equa{p\lra (q,r,s) \text{ Apply } (12) \text{ to get } -p \lra (-r,-q,-s)\\
p\lra (q,r,s) \text{ Apply } (23) \text{ to get } r \lra (-q,p,-s)\\
p\lra (q,r,s) \text{ Apply } (34) \text{ to get } s \lra (-q,-r,p)\\
p\lra (q,r,s) \text{ Apply } (14) \text{ to get } -p \lra (-s,-r,-q).}
\begin{itemize}
\item The set $S$ has $384$ elements. Then each transitive orbit of an element under the action of $S_4$ contains 
$24$ elements. There are $16$ orbits.
\item There are $8$ orbits \emph{(}192 elements satisfying property~\emph{\ref{enum:SymbolNegDet})} 
that arise as compatible symbols associated to concrete four-antipodal point arrangements.
\item Each transitive orbit is the set of all compatible symbols corresponding to one fixed four antipodal pairs of 
points of the point arrangement on the sphere $S^2$ provided one of the symbols in the orbit is 
compatible.
\item Moreover the action of $S_4$ on the set $S$ is free. 
\end{itemize}
\end{theorem}
\begin{proof}
We have $\#(S)=4*3!*2^4=384$. We observe that the action is compatible with the relations
\equa{&(12)(23)(12)=(23)(12)(23),(23)(34)(23)=(34)(23)(34),\\
&(34)(14)(34)=(14)(34)(14),(12)(34)=(34)(12),(14)(23)=(23)(14),\\
&(12)^2=(23)^2=(34)^2=(41)^2=identity.}
So we have an action of the symmetric group $S_4$ on the set $S$ of symbols.
The set of compatible symbols, as a transitive orbit, obtained by the action of $S_4$ on the compatible
symbol $P\lra (y,x,z)$ is precisely the above given $24$ compatible symbols of the standard arrangement
in Section~\ref{sec:SAAS}. Similarly for every transitive orbit if one of the symbols is compatible then all 
the remaining $23$ symbols of the orbit are compatible. The rest of the proof of the theorem is immediate.
\end{proof}
\subsubsection{\bf{The Standard Arrangement and the Dictionary of Line-Cycles}}
Here we associate line cycles to the points of the standard arrangement $\mcl{S}_4$. 
Later we use this as a local dictionary for an antipodal arrangement on $S^2$ to characterize the arrangement up to an isomorphism.

Consider the standard four-antipodal point arrangement $\mcl{S}_4$  given by 
\equa{P_1&=x=[(1,0,0)],-P_1=-x=[(-1,0,0)],\\
P_2&=y=[(0,1,0)],-P_2=-y=[(0,-1,0)],\\
P_3&=z=[(0,1,0)],-P_3=-z=[(0,-1,0)],\\
P_4&=P \text{ a point in } I\operatorname{-}Octant,\\
-P_4&=-P \text{ its antipode point in } VII\operatorname{-}Octant \text{ in }S^2}
The compatible $24$ symbols (an $S_4$ transitive orbit) gives rise to the following dictionary of line cycles using subscripts $\{1,2,3,4\}$ and symbols $\{+,-\}$ at 
each point of the arrangement with $(x,y,z)$ denoting a positively oriented basis of the arrangement.
\equa{&\gt_4^{+}=(213)\text{ at }P_4,\gt_4^{-}=(231)\text{ at }-P_4,\\
&\gt_3^{+}=(142)\text{ at }P_3,\gt_3^{-}=(124)\text{ at }-P_3,\\
&\gt_2^{+}=(341)\text{ at }P_2,\gt_2^{-}=(314)\text{ at }-P_2,\\
&\gt_1^{+}=(243)\text{ at }P_1,\gt_1^{-}=(234)\text{ at }-P_1.}

Now we prove a theorem that given the dictionary of line cycles
there is a unique way to recover back the $24$ compatible
symbols, an $S_4$ orbit of the arrangement, which is compatible with the standard arrangement.

We state the theorem as follows.
\begin{theorem}
\label{theorem:FPSAISO}
Let $\mcl{P}_4=\{\pm P_1,\pm P_2,\pm P_3,\pm P_4\}\subs S^2$ be any four-antipodal point 
arrangement on the sphere. Suppose the line cycles are given by
\equa{&\gt_4^{+}=(213)\text{ at }P_4,\gt_4^{-}=(231)\text{ at }-P_4,\\
&\gt_3^{+}=(142)\text{ at }P_3,\gt_3^{-}=(124)\text{ at }-P_3,\\
&\gt_2^{+}=(341)\text{ at }P_2,\gt_2^{-}=(314)\text{ at }-P_2,\\
&\gt_1^{+}=(243)\text{ at }P_1,\gt_1^{-}=(234)\text{ at }-P_1.}
Then the map $\gd: \mcl{P}_4\lra \mcl{S}_4$ given by 
\equa{\gd:&P_1\lra x,-P_1\lra -x,P_2\lra y,-P_2\lra -y,\\
&P_3\lra z,-P_3\lra -z,P_4\lra P,-P_4\lra -P}
is an isomorphism, that is, it is a convex positive bijection. Also $-\gd$ is an isomorphism.
The $S_4$ invariant set of $24$ compatible symbols are given by
\equa{P_4 &\lra (P_2,P_1,P_3), P_4 \lra (P_1,P_3,P_2),P_4 \lra (P_3,P_2,P_1),\\
P_1&\lra (-P_2,P_4,-P_3),P_1\lra (P_4,-P_3,-P_2),P_1\lra (-P_3,-P_2,P_4),\\
P_3&\lra (-P_1,P_4,-P_2),P_3\lra (P_4,-P_2,-P_1),P_3\lra (-P_2,-P_1,P_4),\\
P_2&\lra (-P_3,P_4,-P_1),P_2\lra (-P_1,-P_3,P_4),P_2\lra (P_4,-P_1,-P_3),\\
-P_1&\lra (-P_4,P_2,P_3),-P_1\lra (P_3,-P_4,P_2),-P_1\lra (P_2,P_3,-P_4),\\
-P_3&\lra (-P_4,P_1,P_2),-P_3\lra (P_1,P_2,-P_4),-P_3\lra (P_2,-P_4,P_1),\\
-P_2&\lra (-P_4,P_3,P_1),-P_2\lra (P_1,-P_4,P_3),-P_2\lra (P_3,P_1,-P_4),\\
-P_4&\lra (-P_1,-P_2,-P_3),-P_4\lra (-P_3,-P_1,-P_2),\\
-P_4&\lra (-P_2,-P_3,-P_1).}
\end{theorem}
\begin{proof}
Let us denote 
\equa{\{P_1&=x,-P_1=-x,P_2=y,-P_2=-y,\\P_3&=z,-P_3=-z,P_4=P,-P_4=-P\}.}
The octant views are given in Figure~\ref{fig:One} based on the point $P$ lying in various octants
with respect to a positively oriented system $(x,y,z)$.
\begin{figure}[h]
	\centering
	\includegraphics[width = 1.0\textwidth]{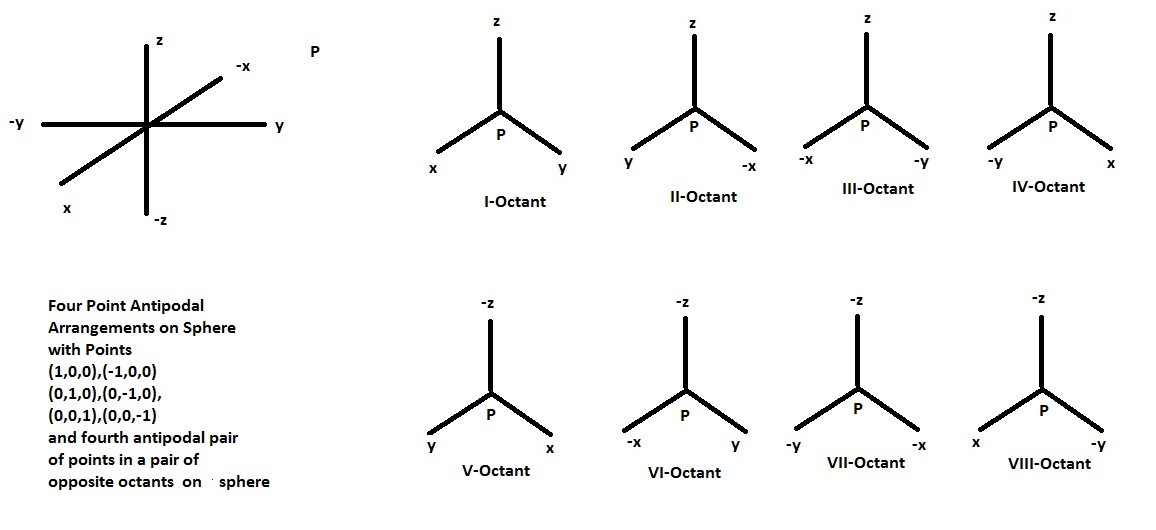}
	\caption{Four Point Arrangements on the Sphere $S^2$}
	\label{fig:One}
\end{figure}
However first we show that $(P_1,P_2,P_3)$ is positively oriented, that is, its determinant is positive
and the symbol $P_4\lra (P_2,P_1,P_3)$ is compatible. A priori we do not know the orientation of $(P_1,P_2,P_3)$
and the compatibility signs of the symbols.

Consider the following choices. $(\pm P_1,\pm P_2,\pm P_3)$. Out of these 
\equ{(P_1,P_2,P_3),(-P_1,-P_2,P_3),(P_1,-P_2,-P_3),(-P_1,P_2,-P_3)}
have the same sign of the determinant and the remaining
\equ{(-P_1,-P_2,-P_3),(P_1,P_2,-P_3),(-P_1,P_2,P_3),(P_1,-P_2,P_3)}
have the same sign of the determinant.

If the second set of determinants are positive then we argue as follows using Theorem~\ref{theorem:S4Action}.
Suppose $P_4\lra (-P_2,-P_1,-P_3)$ is compatible then we have $-P_1\lra (P_2,P_4,P_3)$ is compatible.
Hence $\gt_1^{-}=(243)$ which is invalid.
Suppose $P_4\lra (P_2,P_1,-P_3)$ is compatible then we have $-P_3\lra (-P_1,P_4,-P_2)$ is compatible.
Hence $\gt_3^{-}=(142)$ which is invalid.
Suppose $P_4\lra (P_2,-P_1,P_3)$ is compatible then we have $-P_1\lra (-P_2,P_4,-P_3)$ is compatible.
Hence $\gt_1^{-}=(243)$ which is invalid.
Suppose $P_4\lra (-P_2,P_1,P_3)$ is compatible then we have $-P_2\lra (-P_3,P_4,-P_1)$ is compatible.
Hence $\gt_2^{-}=(341)$ which is invalid.

If the first set of determinants are positive then we argue as follows using Theorem~\ref{theorem:S4Action}.
We have $P_4\lra (-P_2,-P_1,P_3),P_4\lra (-P_2,P_1,-P_3),P_4\lra (P_2,-P_1,-P_3)$ also give invalid line cycles. 
Hence we conclude that $(P_1,P_2,P_3)$ is positively oriented and the symbol $P_4\lra (P_2,P_1,P_3)$
is compatible.

This proves the theorem. We also note that the over all total flip given by 
\equa{\gn: \mcl{P}_4\lra \mcl{S}_4,\gn:&P_1\lra -x,-P_1\lra x,P_2\lra -y,-P_2\lra y,\\
&P_3\lra -z,P_3\lra z,P_4\lra -P,-P_4\lra P}
is also an isomorphism, that is, a convex positive bijection. Using these line cycles we can write down all the $S_4$ invariant set of
$24$ compatible symbols.

\end{proof}
There are other isomorphisms from $\mcl{P}_4$ to $\mcl{S}_4$ as well and below we describe all of them via the automorphism 
group $\Aut(\mcl{S}_4)$.

\subsubsection{\bf{Automorphism Group of the Standard Antipodal Point Arrangement}}
Here we compute the automorphism group of the standard antipodal point arrangement.
\begin{theorem}
Let $\mcl{S}_4\subs S^2$ be the standard arrangement. Then 
\equ{\Aut(\mcl{S}_4)=S_4 \oplus (\Z/2\Z).}
\end{theorem}
\begin{proof}
We have the twenty four compatible symbols of the standard arrangement given in Section~\ref{sec:SAAS}.
If $p\lra (q,r,s)$ is one such compatible symbol then the map 
\equa{\gd:\mcl{S}_4\lra \mcl{S}_4,\gd:&P\lra p,-P\lra -p,y\lra q,-y\lra -q,\\
&x\lra r,-x\lra -r,z\lra s,-z\lra -s}
is an automorphism. We also have if $\gf$ is an automorphism then $-\gf$ is also an automorphism and
moreover either $\gf(P)\lra (\gf(y),\gf(x),\gf(z))$ or $-\gf(P)\lra (-\gf(y),-\gf(x),-\gf(z))$ gives rise to a compatible
symbol and the other one is not a compatible symbol. Hence we get 
\equ{\Aut(\mcl{S}_4)=S_4 \oplus (\Z/2\Z).}
This proves the theorem.
\end{proof}

\subsection{An Isomorphism Theorem for Antipodal Point Arrangements on the Two-Dimensional Sphere 
$S^2$}
~\\
Now we prove an isomorphism theorem for antipodal point arrangements on the sphere $S^2$
which can be generalized to higher dimensions in Theorem~\ref{theorem:LCI}.

\subsubsection{\bf{Localization to Antipodal Point Sub-arrangements}}
~\\
Let $\mcl{P}_n=\{\pm P_1,\pm P_2,\ldots, \pm P_n\}\subs S^2$ be an antipodal point arrangement.
We introduce an equivalence relation $\sim$ on the symmetric group $S_n$ as follows. 
Let $g,h\in S_n$ then $g\sim h$ if $g=h$ or $g=h^{-1}$. This is an equivalence relation with reflexive, symmetric and
transitive properties. The equivalence classes being $[\gt,\gt^{-1}]$. Any element of order at most 
two is an equivalence class containing just one element. Remaining equivalence classes has two elements.
The antipode map $a:S^2\lra S^2$ (reflection about the origin in three dimensions) has a negative determinant. 
The line cycles $\gt_i^{+}$ for $P_i, \gt_i^{-}$ for $-P_i$ associated to a pair of antipodes in $\mcl{P}_n$
are mutually inverses of each other as there is a reflection about the origin
is involved. So we actually obtain 
\equ{(n-1)\operatorname{-}\text{cycles with }\gt_i^{-}= (\gt_i^{+})^{-1}\in S_{n-1}(\{1,2,\ldots,i-1,i+1,\ldots,n\}),1\leq i\leq n.}
Now we consider the local scenario by restricting to just four antipodal pairs.
The restriction map $\mid_{local_A}$ and inverse map $(*)^{-1}$ commutes.
We observe that 
\equa{\text{For any four subset }A&\subs \{1,2,\ldots,n\},\\ (\gt_i^{-})_{\mid_{local_{A}}}&=((\gt_i^{+})^{-1})_{\mid_{local_{A}}}=
((\gt_i^{+})_{\mid_{local_{A}}})^{-1},i\in A}
\subsubsection{\bf{The Main Isomorphism Theorem in Two Dimensions}}
Now we state the theorem as follows.
\begin{theorem}
\label{theorem:MITTD}
The following two assertions hold true.
\begin{enumerate}
\item The line cycles of the antipodal pairs of points of a point arrangement 
$\mcl{P}_n\subs S^2$ determines the collection of local $S_4\operatorname{-}$invariant
set of compatible symbols for every four subset of antipodal pairs of points in $\mcl{P}_n$.
\item Let 
$\mcl{P}^1_n=\{\pm P^1_1, \pm P^1_2,\ldots,\pm P^1_n\},\mcl{P}^2_n=\{\pm P^2_1,\pm P^2_2,\ldots,\pm P^2_n\}$
be two point arrangements. Let $(\gt_i^{+})_j$ be the line cycle associated to $P^j_i$ and 
$(\gt_i^{-})_j$ be the line cycle associated to $-P^j_i$ for $j=1,2,1\leq i\leq n$.
There exists a convex positive bijection (an isomorphism) $\gd:\mcl{P}^1_n\lra \mcl{P}_n^2$ if and only if there exist 
\begin{itemize}
\item a permutation $\gp\in S_n$ and
\item a sign vector $\gm=(\gm(1),\gm(2),\ldots,\gm(n))\in (\Z/2\Z)^n=\{\pm 1\}^n$
\end{itemize}
with the property that 
\begin{enumerate}
\item \label{enum:NoFlip}
either 
\equa{&(\gt_{\gp(i)}^{[\gm(i)*(+)]})_2=\gp(\gt_i^{+})_1\gp^{-1},\\
&(\gt_{\gp(i)}^{[\gm(i)*(-)]})_2=\gp(\gt_i^{-})_1\gp^{-1},
1\leq i\leq n}
\item \label{enum:Flip}
or an overall total flip (here we can choose $-\gm$ in place of $\gm$)
\equa{&(\gt_{\gp(i)}^{[\gm(i)*(+)]})_2=[\gp(\gt_i^{+})_1\gp^{-1}]^{-1}=\gp(\gt_i^{-})_1\gp^{-1},\\
&(\gt_{\gp(i)}^{[\gm(i)*(-)]})_2=[\gp(\gt_i^{-})_1\gp^{-1}]^{-1}=\gp(\gt_i^{+})_1\gp^{-1},
1\leq i\leq n}
\end{enumerate}
where 
\begin{itemize}
\item $[\gm(i)*(+ )]=+,[\gm(i)*(-)]=-$ if $\gm(i)=+$. 
\item $[\gm(i)*(+)]=-,[\gm(i)*(-)]=+$ if $\gm(i)=-$.
\end{itemize}
\end{enumerate}
\end{theorem}
\begin{proof}
We prove the second assertion first.
Suppose $\gd:\mcl{P}^1_n\lra \mcl{P}_n^2$ is an orientation preserving or orientation reversing isomorphism as it is an isomorphism using Theorem~\ref{theorem:PreservingReversing}. Then the permutation $\gp$ and the signed vector $\gm$
are defined by the equation
\equ{\gd(P_i)=\gm(i)P_{\gp(i)},\gd(-P_i)=-\gm(i)P_{\gp(i)},1\leq i\leq n.}
Now with this definition of $\gp,\gm$ the property~\ref{enum:NoFlip} is satisfied if $\gd$ is orientation preserving and the property~\ref{enum:Flip} is satisfied if it is orientation reversing.
If we choose for $\gm$ the following definition
\equ{\gd(P_i)=-\gm(i)P_{\gp(i)},\gd(-P_i)=\gm(i)P_{\gp(i)},1\leq i\leq n}
then $\gp,\gm$ satisfies the property~\ref{enum:Flip} if $\gd$ is orientation preserving and the property~\ref{enum:NoFlip} is satisfied if it is orientation reversing. This proves one way implication. 

Now we prove the other way implication where we are given the permutation $\gp$ and the signed vector $\gm$ and changing $\gm$
to $-\gm$ if necessary we assume that the property~\ref{enum:NoFlip} holds. 
First we localize to any two corresponding four-antipodal point arrangements
\equ{\{\pm P_i,\pm P_j,\pm P_k,\pm P_l\},\{\pm P_{\gp(i)},\pm P_{\gp(j)},\pm P_{\gp(k)},\pm P_{\gp(l)}\}.}
Since property~\ref{enum:NoFlip} holds and the restriction map and the inverse map commutes with respect to localization 
there is an isomorphic way to identify these two arrangements
using local line cycles via the local chart as the standard arrangement
$\mcl{S}_4$ using Theorem~\ref{theorem:FPSAISO}. Using this chart we conclude that locally there exists an isomorphism
of the four-antipodal point arrangements given by 
\equa{\gd:&P_i\lra \gm(i)P_{\gp(i)},-P_i\lra -\gm(i)P_{\gp(i)},P_j\lra \gm(j)P_{\gp(j)},-P_j\lra -\gm(j)P_{\gp(j)}\\
&P_k\lra \gm(k)P_{\gp(k)},-P_k\lra -\gm(k)P_{\gp(k)},P_l\lra \gm(l)P_{\gp(l)},-P_l\lra -\gm(l)P_{\gp(l)}}
These local isomorphisms patch up and extend uniquely to an isomorphism defined as 
\equ{\gd(P_i)=\gm(i)P_{\gp(i)},\gd(-P_i)=-\gm(i)P_{\gp(i)},1\leq i\leq n.}
We also observe that $-\gd:\mcl{P}^1_n\lra \mcl{P}^2_n$ is an isomorphism. Moreover it is either orientation preserving or orientation reversing.
This proves the isomorphism theorem in two dimensions.

Now we prove the first assertion. The local cycles of four-antipodal subarrangements determine the $S_4$ invariant set of $24$
compatible symbols using Theorem~\ref{theorem:FPSAISO}. 
Hence the first assertion follows and we can write down all the compatible symbols of the given arrangement.
\end{proof}
\section{\bf{Examples of two Non-isomorphic Normal Systems in Three Dimensions over Rationals: Revisited}}
\label{sec:NonIsoNS}
Consider the normal systems whose associated sets of antipodal vectors are given by   
\equ{\mcl{U}_1=\{\pm u_i\mid 1\leq i \leq 6\},\mcl{U}_2=\{\pm v_i\mid 1\leq i \leq 6\}}
with $\mcl{U}_1\cap \mcl{U}_2=\{\pm u_1,\pm u_2,\pm u_3,\pm u_4,\pm u_5\}=\{\pm v_1,\pm v_2,\pm v_3,\pm v_4,\pm 
v_5\}$ where
\equa{&u_1=(1,0,0)=v_1,u_2=(0,1,0)=v_2,u_3=(0,0,1)=v_3,\\
&u_4=\big(\frac 13,\frac 23,\frac 23\big)=v_4, u_5=\big(\frac 19,\frac 49,\frac 89\big)=v_5,
u_6=\big(\frac 6{11},\frac 6{11},\frac 7{11}\big),v_6=\big(\frac 2{11},\frac 6{11},\frac 9{11}\big).}
Here below we find out line cycles of each point with respect to the given notation.

We have proved that these two are non-isomorphic normal systems by associating graphs of compatible pairs 
mentioned in C.~P.~Anil~Kumar~\cite{NRHA}.
For example, from the $15$ equations below for $\mcl{U}_1$ the vertex $\{-u_1,u_2\}$ has degree one and 
is only compatible with $\{u_4,-u_6\}$ (equation $(5)$ in the first set). From the $15$ equations below for $\mcl{U}_2$ we observe that there
is no vertex of degree one as we observe that if a vertex of the associated graph of
compatible pairs has a positive degree then the degree is at least two. 

Now we mention the following $\binom{6}{4}=15$ equations for $\mcl{U}_1$.
\begin{enumerate}
\item $3u_4= u_1+2u_2+2u_3=(1,2,2)$.
\item $9u_5=u_1+4u_2+8u_3=(1,4,8)$.
\item $11u_6=6u_1+6u_2+7u_3=(6,6,7)$.
\item $12u_4=3u_1+4u_2+9u_5=(4,8,8)$.
\item $5u_1+21u_4=2u_2+22u_6=(12,14,14)$.
\item $88u_6=41u_1+20u_2+63u_5=(48,48,56)$.
\item $u_1+9u_5=4u_3+6u_4=(2,4,8)$.
\item $11u_6=3u_1+u_3+9u_4=(6,6,7)$.
\item $9u_1+9u_5=10u_3+22u_6=(12,12,24)$.
\item $9u_5=2u_2+6u_3+3u_4=(1,4,8)$.
\item $18u_4=6u_2+5u_3+11u_6=(6,12,12)$.
\item $54u_5=18u_2+41u_3+11u_6=(6,24,48)$.
\item $44u_6=13u_1+30u_4+9u_5=(24,24,28)$
\item $123u_4=26u_2+45u_5+66u_6=(41,82,82)$.
\item $13u_3+27u_4=27u_5+11u_6=(9,18,31)$.
\end{enumerate}
Then we have by actual computation the line cycles are given as 
\equa{&(\gt_1^{+})_1=(24653)  \text{ at }u_1,(\gt_1^{-})_1=(23564) \text{ at }-u_1,\\
&(\gt_2^{+})_1=(13546) \text{ at }u_2,(\gt_2^{-})_1=(16453) \text{ at }-u_2,\\
&(\gt_3^{+})_1=(16452) \text{ at }u_3,(\gt_3^{-})_1=(12546) \text{ at }-u_3,\\
&(\gt_4^{+})_1=(15326) \text{ at }u_4,(\gt_4^{-})_1=(16235) \text{ at }-u_4,\\
&(\gt_5^{+})_1=(16432) \text{ at }u_5,(\gt_5^{-})_1=(12346) \text{ at }-u_5,\\
&(\gt_6^{+})_1=(15324) \text{ at }u_6,(\gt_6^{-})_1=(14235) \text{ at }-u_6.}

Now we mention the following $\binom{6}{4}=15$ equations for $\mcl{U}_2$.
\begin{enumerate}
\item $3v_4=v_1+2v_2+2v_3=(1,2,2)$.
\item $9v_5=v_1+4v_2+8v_3=(1,4,8)$.
\item $11v_6=2v_1+6v_2+9v_3=(2,6,9)$.
\item $12v_4=3v_1+4v_2+9v_5=(4,8,8)$.
\item $27v_4=5v_1+6v_2+22v_6=(9,18,18)$.
\item $88v_6=7v_1+12v_2+81v_5=(16,48,72)$.
\item $v_1+9v_5=4v_3+6v_4=(2,4,8)$.
\item $v_1+11v_6=3v_3+9v_4=(3,6,9)$.
\item $v_1+27v_5=6v_3+22v_6=(4,12,24)$.
\item $9v_5=2v_2+6v_3+3v_4=(1,4,8)$.
\item $11v_6=2v_2+5v_3+6v_4=(2,6,9)$.
\item $18v_5=2v_2+7v_3+11v_6=(2,8,16)$.
\item $v_1+44v_6=18v_4+27v_5=(9,24,36)$.
\item $66v_6=2v_2+21v_4+45v_5=(12,6,54)$.
\item $v_3+11v_6=3v_4+9v_5=(2,6,10)$.
\end{enumerate}

Then we have by actual computation the line cycles are given as 
\equa{&(\gt_1^{+})_2= (24653) \text{ at }v_1,(\gt_1^{-})_2= (23564) \text{ at }-v_1,\\
&(\gt_2^{+})_2= (13564) \text{ at }v_2,(\gt_2^{-})_2= (14653) \text{ at }-v_2,\\
&(\gt_3^{+})_2= (14652) \text{ at }v_3,(\gt_3^{-})_2= (12564) \text{ at }-v_3,\\
&(\gt_4^{+})_2= (16532) \text{ at }v_4,(\gt_4^{-})_2= (12356) \text{ at }-v_4,\\
&(\gt_5^{+})_2= (14632) \text{ at }v_5,(\gt_5^{-})_2= (12364) \text{ at }-v_5,\\
&(\gt_6^{+})_2= (14532) \text{ at }v_6,(\gt_6^{-})_2= (12354) \text{ at }-v_6.}

As an example we just show that $(\gt_1^{+})_1=(24653)$. For this we observe the following. 
\begin{enumerate}
\item $u_1=-2u_2-2u_3+3u_4$ and $\Det\matthree{0}{-1}{0}{1}{2}{2}{0}{0}{-1}<0$. Hence the compatible cycle is $(243)$.
\item $u_1=-4u_2-8u_3+9u_5$ and $\Det\matthree{0}{-1}{0}{1}{4}{8}{0}{0}{-1}<0$. Hence the compatible cycle is $(253)$.
\item $6u_1=-6u_2-7u_3+11u_6$ and $\Det\matthree{0}{-1}{0}{6}{6}{7}{0}{0}{-1}<0$. Hence the compatible cycle is $(263)$.
\item $3u_1=-4u_2+12u_4-9u_5$ and $\Det\matthree{0}{-1}{0}{1}{2}{2}{-1}{-4}{-8}<0$. Hence the compatible cycle is $(245)$.
\item $5u_1=2u_2-21u_4+22u_6$ and $\Det\matthree{0}{1}{0}{-1}{-2}{-2}{6}{6}{7}<0$. Hence the compatible cycle is $(246)$.
\item $41u_1=-20u_2-63u_5+88u_6$ and $\Det\matthree{0}{1}{0}{6}{6}{7}{-1}{-4}{-8}<0$. Hence the compatible cycle is $(265)$.
\item $u_1=4u_3+6u_4-9u_5$ and $\Det\matthree{0}{0}{1}{1}{2}{2}{-1}{-4}{-8}<0$. Hence the compatible cycle is $(345)$.
\item $3u_1=-u_3-9u_4+11u_6$ and $\Det\matthree{0}{0}{-1}{-1}{-2}{-2}{6}{6}{7}<0$. Hence the compatible cycle is $(346)$.
\item $9u_1=10u_3-9u_5+22u_6$ and $\Det\matthree{0}{0}{1}{6}{6}{7}{-1}{-4}{-8}<0$. Hence the compatible cycle is $(365)$.
\item $13u_1=-30u_4-9u_5+44u_6$ and $\Det\matthree{-1}{-2}{-2}{6}{6}{7}{-1}{-4}{-8}<0$. Hence the compatible cycle is $(465)$.
\end{enumerate}
Hence the $5\operatorname{-}$cycle compatible with all the ten $3\operatorname{-}$cycles is given by $(24653)$.
The computation of the rest of the five cycles is similar. 
\section{\bf{Antipodal Point Arrangements on Higher Dimensional Spheres and Classification of Normal Systems}}
\label{sec:APAHDSCNS}
Here we mainly associate combinatorial invariants to antipodal point arrangements to classify them and hence classify
the normal systems up to an isomorphism. These combinatorial invariants turn out to be oriented cycles of
points of the orthogonally projected arrangements along small sub-arrangements. We begin with the required definitions.
\begin{defn}[Antipodal Point Arrangement on the $k\operatorname{-}$Sphere $S^k$]
\label{defn:kDAPA}
~\\
We say a set $\mcl{P}_n=\{\pm P_1,\pm P_2,\ldots,\pm P_n\}\subs S^k$ of points is a point 
arrangement on the sphere if for any $1\leq i_1<i_2<\ldots<i_{k+1}\leq n$ the points 
\equ{P_{i_1},P_{i_2},\ldots,P_{i_{k+1}}} are linearly independent.
\end{defn}
\begin{defn}[Isomorphism Between two Antipodal Point Arrangements on the $k\operatorname{-}$Sphere $S^k$]
\label{defn:kDIso}
~\\
Two point arrangements 
\equ{\mcl{P}_n=\{\pm P_1,\pm P_2,\ldots,\pm P_n\},\mcl{Q}_m=\{\pm Q_1,\pm Q_2,\ldots,\pm Q_m\} \subs 
S^k} 
are isomorphic if $n=m$ and there is a bijection $\gf:\mcl{P}_n \lra \mcl{Q}_n$ between the two sets such that the 
following occurs.
\begin{itemize}
\item $\gf(-A)=-\gf(A)$ for all $A\in \mcl{P}_n$.
\item  for any $A, A_{i_l}\in \mcl{P}_n,1\leq l\leq k+1,A$ is a positive combination of $A_{i_l},1\leq l\leq k+1$
if and only if $\gf(A)$ is a positive combination of $\gf(A_{i_l}),1\leq l\leq k+1$.
\end{itemize}
We sometimes also say that the isomorphism $\gf$ is a convex positive bijection.
We say $\gf$ is orientation preserving if for any $k+1$ points $A_{i_l}\in \mcl{P}_n,1\leq l\leq k+1$ the ordered tuple $(A_{i_1},\ldots,A_{i_l})$ has positive determinant if and only if the ordered tuple $(\gf(A_{i_1}),\ldots,\gf(A_{i_l}))$ has positive determinant. We say $\gf$ is orientation reversing if for any $k+1$ points $A_{i_l}\in \mcl{P}_n,1\leq l\leq k+1$ the ordered tuple $(A_{i_1},\ldots,A_{i_l})$ has positive determinant if and only if the ordered tuple $(\gf(A_{i_1}),\ldots,\gf(A_{i_l}))$ has negative determinant.
\end{defn}

\begin{theorem}
\label{theorem:PreservingReversingHigherDim}
Let \equ{\mcl{P}_n=\{\pm P_1,\pm P_2,\ldots,\pm P_n\},\mcl{Q}_n=\{\pm Q_1,\pm Q_2,\ldots,\pm Q_n\} \subs S^k}
be two antipodal point arrangements on the sphere $S^k$. Let $\gd:\mcl{P}_n \lra \mcl{Q}_n$ be an isomorphism. Then it either an orientation preserving isomorphism or an orientation reversing isomorphism.
\end{theorem}
\begin{proof}
The proof is similar to that of Theorem~\ref{theorem:PreservingReversing}.
\end{proof}
\section{\bf{Dimension Reduction and Multiple Orthogonally Projected Antipodal Arrangements along Small Sub-arrangements}}
~\\
We begin with a definition.
\begin{defn}[Orthogonally Projected Antipodal Point Arrangements]
~\\
Let $\mcl{P}_n=\{\pm P_1,\ldots ,\pm P_n\}$ be an antipodal point arrangement in the $k\operatorname{-}$dimensional sphere 
$S^k$. Let $\mcl{A}=\{\pm P_{i_1},\pm P_{i_2},\ldots,\pm P_{i_r}\}$ $\subs \mcl{P}_n$
be an antipodal point sub-arrangement with $1\leq r\leq k-2$. 
We can orthogonally project using $Q^{\mcl{A}}$ the sub-arrangement $\mcl{P}_n\bs \mcl{A}$ to the space orthogonal 
to the space spanned by vectors of $\mcl{A}$ to obtain an antipodal point arrangement 
\equ{\mcl{P}_{n-r}^{\mcl{A}}=\{P^{\mcl{A}}_j=Q^{\mcl{A}}(P_j),-P^{\mcl{A}}_j=Q^{\mcl{A}}(-P_j)
\mid 1\leq j\leq n,j\neq i_l,1\leq l\leq r\}} 
in the $(k-r)$-dimensional sphere $S^{k-r}$.
\end{defn}
Now we prove a theorem on the signs.
\begin{theorem}[Sign of the Combination does not change after Projection]
\label{theorem:SCSP}
~\\
Let $\mcl{P}_n=\{\pm P_1,\ldots ,\pm P_n\}$ be an antipodal point arrangement in the $k\operatorname{-}$dimensional sphere 
$S^k$. Let $\mcl{A}=\{\pm P_n\}$. Let $\mcl{P}^{\mcl{A}}_{n-1}=\{\pm P_1^{\mcl{A}},\pm P_2^{\mcl{A}},\ldots,\pm P_{n-1}^{\mcl{A}}\}$ denote the projected arrangement.
Suppose \equ{P_i=[(x^i_1,x^i_2,\ldots,x^i_k,x^i_{k+1})],-P_i=-[(x^i_1,x^i_2,\ldots,x^i_k,x^i_{k+1})],1\leq i\leq n.} 
Suppose we have 
\equ{(x^j_1,x^j_2,\ldots,x^j_k,x^j_{k+1})=\us{l=1}{\os{k}{\sum}}\gl_l(x^{i_l}_1,x^{i_l}_2,\ldots,x^{i_l}_k,x^{i_l}_{k+1})
+\gl_n(x^n_1,x^n_2,\ldots,x^n_k,x^n_{k+1})}
for some $j\nin\{i_1,i_2,\ldots,i_k,n\}, \gl_l,\gl_n\in \mbb{R}^{*}$.
Suppose we have 
\equ{P_j^{\mcl{A}}=Q^{\mcl{A}}(x^j_1,x^j_2,\ldots,x^j_k,x^j_{k+1})=\us{l=1}{\os{k}{\sum}}\gb_lP^{\mcl{A}}_{i_l}=\us{l=1}{\os{k}{\sum}}\gb_lQ^{\mcl{A}}(x^{i_l}_1,x^{i_l}_2,\ldots,x^{i_l}_k,x^{i_l}_{k+1})}
then \equ{sign(\gl_l)=sign(\gb_l),1\leq l\leq k.}
\end{theorem}
\begin{proof}
This theorem is immediate.
\end{proof}

Now we prove an isomorphism theorem about signs for antipodal point arrangements on spheres $S^k$.
\begin{theorem}[An isomorphism theorem]
\label{theorem:SL}
~\\
Let $\mcl{P}^j_n=\{\pm P^j_1,\pm P^j_2,\ldots,\pm P^j_n\}$ be two antipodal point arrangements in $S^k$ for $j=1,2$. 
Let \equ{\pm P^j_{i_1},\pm P^j_{i_2},\ldots,\pm P^j_{i_{k+1}},\pm P^j_l}
be $k+2$ antipodal pairs of points of choice of the arrangement. Let $A=\{i_1,i_2,\ldots,i_{k+1}\}$. 
With respect to the set $A$ let \equ{P^j_l=\us{r=1}{\os{m}{\sum}}(\gl^l_{i_r})^j_AP^j_{i_r},j=1,2.}
Suppose we have \equ{sign((\gl^l_{i_r})^1_A)=sign((\gl^l_{i_r})^2_A)\text{ for any such choice.}} Then the map 
\equ{\gd:P^1_i\lra P^2_i,\gd:-P^1_i\lra -P^2_i, 1\leq i\leq n}
is an isomorphism of antipodal point arrangements $\mcl{P}^j_n$.
\end{theorem}
\begin{proof}
This theorem is immediate and $\gd$ is a convex positive bijection.
\end{proof}
\subsection{Line Cycle Invariants Associated to Points of the Projected Arrangements}
~\\
Let $\mcl{P}_n=\{\pm P_1,\pm P_2,\ldots,\pm P_n\}$ be an antipodal
point arrangement in $S^k$. Let 
\equ{\mcl{A}=\{\pm P_{j_1},\pm P_{j_2},\ldots,\pm P_{j_{k-2}}\}\subs \mcl{P}_n} be a subset of cardinality $k-2$. Then 
consider the projected arrangement $\mcl{P}^{\mcl{A}}_{n-k+2}\subs S^2$. For various choices of $\mcl{A}$, these arrangements on the
two-dimensional spheres give rise to clockwise oriented line cycles at each point of $\mcl{P}^{\mcl{A}}_{n-k+2}$
denoted as follows:
\equa{(\gt_j^{+})^{\mcl{A}}&\in S_{n-k+1}\big(\{1,2,\ldots,n\}\bs \{j_1,j_2,\ldots,j_{k-2},j\}\big) 
\text{ at }P^{\mcl{A}}_j,\\
(\gt_j^{-})^{\mcl{A}}&\in S_{n-k+1}\big(\{1,2,\ldots,n\}\bs \{j_1,j_2,\ldots,j_{k-2},j\}\big) 
\text{ at }-P^{\mcl{A}}_j}
both of which are $(n-k+1)\operatorname{-}$cycles which are mutual inverses of each other for $1\leq j\leq n,j\neq j_l,1\leq l\leq k-2$.

\section{\bf{The Main Theorem}}
\label{sec:FMT}
Now we prove the isomorphism theorem for the line cycle invariants.
\begin{thmOmega}[Main Theorem]
\namedlabel{theorem:LCI}{$\Gom$}
~\\
Let $\mcl{P}_n=\{\pm P_1,\pm P_2,\ldots,\pm P_n\}$ be an antipodal point 
arrangement on $S^k$. The line cycle invariants of antipodal pairs 
\equ{P_j,-P_j,1\leq j\leq n} given by mutually inverse cycles
\equa{(\gt_j^{+})^{\mcl{A}}&\in S_{n-k+1}\big(\{1,2,\ldots,n\}\bs \{j_1,j_2,\ldots,j_{k-2},j\}\big) 
\text{ at }P^{\mcl{A}}_j,\\
(\gt_j^{-})^{\mcl{A}}&\in S_{n-k+1}\big(\{1,2,\ldots,n\}\bs \{j_1,j_2,\ldots,j_{k-2},j\}\big) 
\text{ at }-P^{\mcl{A}}_j}
after projection along the small sub-arrangement 
\equ{\mcl{A}=\{\pm P_{j_1},\pm P_{j_2},\ldots,\pm P_{j_{k-2}}\}\subs \mcl{P}_n\bs\{\pm P_j\}} 
given by 
\equ{\mcl{P}_{n-k+2}^{\mcl{A}}=\{\pm P_i^{\mcl{A}}\mid i\in \{1,2,\ldots,n\}\bs \{j_1,j_2,\ldots,j_{k-2}\}\}} 
for all such possible choices of $\mcl{A}$ determine the 
antipodal point arrangement up to an isomorphism, that is,   
for $l=1,2$ let  
\equ{\mcl{P}^l_n=\{\pm P^l_1,\pm P^l_2,\ldots,\pm P^l_n\}\subs S^k} 
be two antipodal point arrangements with the line cycle invariants of antipodal pairs 
\equ{P^l_j,-P^l_j,1\leq j\leq n} given by mutually inverse cycles
\equa{(\gt_j^{+})^{\mcl{A}_l}_l&\in S_{n-k+1}\big(\{1,2,\ldots,n\}\bs \{j_1,j_2,\ldots,j_{k-2},j\}\big) 
\text{ at }(P^l_j)^{\mcl{A}_l},\\
(\gt_j^{-})^{\mcl{A}_l}_l&\in S_{n-k+1}\big(\{1,2,\ldots,n\}\bs \{j_1,j_2,\ldots,j_{k-2},j\}\big) 
\text{ at }-(P^l_j)^{\mcl{A}_l}}
after projection along the small sub-arrangement 
\equ{\mcl{A}_l=\{\pm P^l_{j_1},\pm P^l_{j_2},\ldots,\pm P^l_{j_{k-2}}\}\subs \mcl{P}^l_n\bs\{\pm P^l_j\},}
then they are isomorphic by a convex positive bijection if and only if there exist  
\begin{enumerate}
\item a permutation $\gp \in S_n$ and 
\item a sign vector $\gm\in (\Z/2\Z)^n=\{\pm \}^n$
\end{enumerate}
such that for all invariant line cycles either 
\equa{(\gt^{[\gm(j)*(+)]}_{\gp(j)})^{\gp(\mcl{A}_1)}_2 = \gp(\gt^{+}_j)^{\mcl{A}_1}_1\gp^{-1},\\
\text{ which is equivalent to }(\gt^{[\gm(j)*(-)]}_{\gp(j)})^{\gp(\mcl{A}_1)}_2 = \gp(\gt^{-}_j)^{\mcl{A}_1}_1\gp^{-1}}
holds or with a total flip the following holds. \emph{(}Also we could flip the sign of $\gm$\emph{)}
\equa{(\gt^{[\gm(j)*(-)]}_{\gp(j)})^{\gp(\mcl{A}_1)}_2 = \gp(\gt^{+}_j)^{\mcl{A}_1}_1\gp^{-1},\\
\text{ which is equivalent to }(\gt^{[\gm(j)*(+)]}_{\gp(j)})^{\gp(\mcl{A}_1)}_2 = \gp(\gt^{-}_j)^{\mcl{A}_1}_1\gp^{-1}}
where 
\begin{enumerate}
\item $\gp(\mcl{A}_1)=\{\pm P^2_{\gp(j_1)},\pm P^2_{\gp(j_2)},\ldots,\pm P^2_{\gp(j_{k-2})}\}$.
\item \begin{itemize}
\item $[\gm(j)*(+)]=+,[\gm(j)*(-)]=-$ if $\gm(j)=+$.
\item $[\gm(j)*(+)]=-,[\gm(j)*(-)]=+$ if $\gm(j)=-$.
\end{itemize}
\end{enumerate}
\end{thmOmega}
\begin{proof}
	A convex positive bijection is either orientation preserving or orientation reversing using Theorem~\ref{theorem:PreservingReversingHigherDim}.
	The orientation preserving or the orientation reversing convex positive bijection $\gd$ and the permutation $\gp\in S_n$ with a sign vector $\gm=(\gm(1),\ldots,\gm(n))\in (\Z/2\Z)^n=\{\pm \}^n$ are related by the following equations.
	\equ{\text{For }1\leq i\leq n,\gd(P^1_i)=\gm(i)P^2_{\gp(i)}.}
	We can assume if necessary and without loss of generality that $\gp$ is trivial and $\gm=(+,\ldots,+)$ so that $\gd(P^1_i)=P^2_i,1\leq i\leq n$.
	If $\gd$ is orientation preserving then we have 
	\equ{(\gt^{+}_{j})^{\mcl{A}_2}_2 = (\gt^{+}_j)^{\mcl{A}_1}_1 \text{ and }(\gt^{-}_{j})^{\mcl{A}_2}_2 = (\gt^{-}_j)^{\mcl{A}_1}_1.}
	If $\gd$ is orientation reversing then we have 	
	\equ{(\gt^{-}_{j})^{\mcl{A}_2}_2 = (\gt^{+}_j)^{\mcl{A}_1}_1 \text{ and }(\gt^{+}_{j})^{\mcl{A}_2}_2 = (\gt^{-}_j)^{\mcl{A}_1}_1.}
	Here \equ{\mcl{A}_l=\{\pm P^l_{j_1},\pm P^l_{j_2},\ldots,\pm P^l_{j_{k-2}}\}\subs \mcl{P}^l_n\bs\{\pm P^l_j\},l=1,2.}
	
To determine the arrangement up to an isomorphism we do the following. For $r=1,2$ let \equ{P^r_{i_1},P^r_{i_2},\ldots,P^r_{i_{k+1}},P^r_l}
be $k+2$ points of the arrangements. With respect to $A=\{i_1,\ldots,i_{k+1}\}$ define the coefficients $(\gl^l_{i_j})^1_A,(\gl^l_{i_j})^2_A$ by  
letting \equ{P^1_l=\us{j=1}{\os{k+1}{\sum}}(\gl^l_{i_j})^1_AP^1_{i_j},
	P^2_l=\us{j=1}{\os{k+1}{\sum}}(\gl^l_{i_j})^2_AP^2_{i_j}.}
To determine the arrangement up to an isomorphism using Theorem~\ref{theorem:SL} 
we need to determine the signs of $(\gl^l_{i_j})^1_A,(\gl^l_{i_j})^2_A,1\leq j \leq k+1$ 
using the combinatorial invariants.

Now when $k=2$ we know that the line cycles give rise to the compatible set of $S_4\operatorname{-}$invariant set of $24$
symbols locally for all the sub-arrangements using Theorem~\ref{theorem:MITTD}. 
These determine the signs and hence the antipodal point arrangement on $S^2$ is determined 
up to an isomorphism.

Now we consider a general value of $k$. Now using various orthogonal projections along subsets of $\{1,\ldots,n\}$ 
and repeated application of Theorem~\ref{theorem:SCSP} we can recover the signs of the coefficients $(\gl^l_{i_j})^1_A,(\gl^l_{i_j})^2_A$ 
from the combinatorial line cycle invariants. Now we use Theorem~\ref{theorem:SL}.

The rest of Theorem~\ref{theorem:LCI} also follows.
\end{proof}

\section{\bf{Open Questions: The Enumeration Problem and The Problem of a Complete List of Representatives}}
\label{sec:ENUMREP}
We have solved the classification problem of isomorphism classes of normal systems in any dimension. Now we mention the two questions which are still open.

\begin{ques}
\label{ques:ENUMREP}
~\\
Let $n,m$ be positive integers.
\begin{enumerate}
\item (Enumeration Problem): Enumerate the isomorphism classes of normal systems in $\mbb{R}^m$
of cardinality $n$.
\item (Representation Problem): Construct a complete list of representatives for the list of
isomorphism classes of normal systems in $\mbb{R}^m$ of cardinality $n$.
\end{enumerate}
\end{ques}

\bibliographystyle{abbrv}
\def\cprime{$'$}

\end{document}